\documentclass[letterpaper, 10 pt, conference]{ieeeconf}  

\usepackage{amsthm}

\usepackage[utf8]{inputenc}
\usepackage{url,bm}
\usepackage{graphics}
\usepackage{amsmath}
\usepackage{amssymb}
\usepackage{tikz}
\usepackage{pgfplots}
\pgfplotsset{compat=newest}
\usetikzlibrary{pgfplots.groupplots,matrix}

\usepackage{algorithm}
\usepackage[noend]{algpseudocode}
\usepackage{cite}
\usepackage{mathrsfs}
\usepackage[unicode=true,pdfusetitle, bookmarks=true,bookmarksnumbered=false,bookmarksopen=false, breaklinks=false,pdfborder={0 0 0},backref=false,colorlinks=true,linkcolor=blue]{hyperref}
\usepackage{enumerate}
\usepackage[mathscr]{euscript}
 \usepackage[all]{hypcap} 
\usepackage{cleveref}
\usepackage{color}

\newcommand{\bld}[1]{\boldsymbol{#1}}

\newcommand{\ri}{\mathop{\rm ri}\nolimits}

\newcommand{\dom}{\mathop{\rm dom}\nolimits}

\newcommand{\blkdiag}{\mathop{\rm blkdiag}\nolimits}

\newcommand{\prox}{\mathop{\rm prox}\nolimits}

\newcommand{\minimize}{\operatorname{minimize}}  
  
\renewcommand{\Re}{{\rm{I\!R}} }
\newcommand{\Rinf}{\overline{\rm I\!R}}
\newcommand{\Nn}{{\rm{I\!N}}}

\newcommand{\zer}{\mathop{\rm zer}\nolimits}
\newcommand{\gra}{\mathop{\rm gra}\nolimits}

\DeclareMathOperator*{\argmin}{\arg\!\min}
\newcommand{\stt}{\rm subject\ to}

\newtheorem{thm}{Theorem}
\newtheorem{lem}{Lemma}

\newtheorem{ass}{Assumption}

\newtheorem{deff}{Definition}

\crefname{thm}{Thm.}{Thm.}
\Crefname{thm}{Theorem}{Theorems}
\crefformat{thm}{#2Thm. #1#3}
\Crefformat{thm}{#2Theorem #1#3}
\crefmultiformat{thm}{#2Thm. #1#3}{ and #2#1#3}{, #2#1#3}{ and #2#1#3}
\Crefmultiformat{thm}{#2Theorems #1#3}{ and #2#1#3}{, #2#1#3}{ and #2#1#3}

\crefname{prop}{Prop.}{Prop.}
\Crefname{prop}{Proposition}{Propositions}
\crefformat{prop}{#2Prop. #1#3}
\Crefformat{prop}{#2Proposition #1#3}
\crefmultiformat{prop}{#2Prop. #1#3}{ and #2#1#3}{, #2#1#3}{ and #2#1#3}
\Crefmultiformat{prop}{#2Propositions #1#3}{ and #2#1#3}{, #2#1#3}{ and #2#1#3}

\crefname{coro}{Cor.}{Cro.}
\Crefname{coro}{Corollary}{Corollaries}
\crefformat{coro}{#2Cor. #1#3}
\Crefformat{coro}{#2Corollary #1#3}
\crefmultiformat{coro}{#2Cor. #1#3}{ and #2#1#3}{, #2#1#3}{ and #2#1#3}
\Crefmultiformat{coro}{#2Corollaries #1#3}{ and #2#1#3}{, #2#1#3}{ and #2#1#3}

\crefname{lem}{Lem.}{Lem.}
\Crefname{lem}{Lemma}{Lemmas}
\crefformat{lem}{#2Lem. #1#3}
\Crefformat{lem}{#2Lemma #1#3}
\crefmultiformat{lem}{#2Lem. #1#3}{ and #2#1#3}{, #2#1#3}{ and #2#1#3}
\Crefmultiformat{lem}{#2Lemmas #1#3}{ and #2#1#3}{, #2#1#3}{ and #2#1#3}

\crefname{algorithm}{Alg.}{Alg.}
\Crefname{algorithm}{Algorithm}{Algorithms}
\crefformat{algorithm}{#2Alg. #1#3}
\Crefformat{algorithm}{#2Algorithm #1#3}
\crefmultiformat{algorithm}{#2Alg. #1#3}{ and #2#1#3}{, #2#1#3}{ and #2#1#3}
\Crefmultiformat{algorithm}{#2Algorithms #1#3}{ and #2#1#3}{, #2#1#3}{ and #2#1#3}

\crefname{rem}{Rem.}{Rem.}
\Crefname{rem}{Remark}{Remarks}
\crefformat{rem}{#2Rem. #1#3}
\Crefformat{rem}{#2Remark #1#3}
\crefmultiformat{rem}{#2Rem. #1#3}{ and #2#1#3}{, #2#1#3}{ and #2#1#3}
\Crefmultiformat{rem}{#2Remarks #1#3}{ and #2#1#3}{, #2#1#3}{ and #2#1#3}

\crefname{ass}{Ass.}{Ass.}
\Crefname{ass}{Assumption}{Assumption}
\Crefformat{ass}{#2Assumption #1#3}
\Crefmultiformat{ass}{#2Assumptions #1#3}{ and #2#1#3}{, #2#1#3}{ and #2#1#3}

\crefname{deff}{Def.}{Def.}
\Crefname{deff}{Definition}{Definition}
\crefformat{deff}{#2Def. #1#3}
\Crefformat{deff}{#2Definition #1#3}

\Crefname{section}{Section}{Sections}
\Crefformat{section}{#2Section #1#3}
\Crefmultiformat{section}{#2Sections #1#3}{ and #2#1#3}{, #2#1#3}{ and #2#1#3}

\Crefname{enumi}{Assumption}{Assumptions}
\Crefformat{enumi}{Assumption #21(#1)#3}
\Crefmultiformat{enumi}{Assumptions (#2#1#3)}{ and (#2#1#3)}{, (#2#1#3)}{ and (#2#1#3)}
\crefrangeformat{enumi}{Assumptions ~(#3#1#4) to~(#5#2#6)}

\Crefname{figure}{Figure}{Figures}
\Crefformat{figure}{#2Figure #1#3}

\IEEEoverridecommandlockouts                              
\title{\LARGE \bf
New Primal-Dual Proximal Algorithm for Distributed Optimization
}

\author{Puya Latafat, Lorenzo Stella, Panagiotis Patrinos
\thanks{P. Latafat  is with IMT School for Advanced Studies Lucca, Piazza San Francesco 19, 55100 Lucca, Italy; Email: {\tt\small puya.latafat@imtlucca.it}}
\thanks{L. Stella  and P. Patrinos are with the Department of Electrical Engineering (ESAT-STADIUS) and Optimization in Engineering Center (OPTEC), KU Leuven, Kasteelpark  Arenberg 10, 3001 Leuven-Heverlee, Belgium; Emails: {\tt\small lorenzo.stella@esat.kuleuven.be, panos.patrinos@esat.kuleuven.be}} 
}

\begin{document}

\maketitle
\thispagestyle{empty}
\pagestyle{empty}

\begin{abstract}
We consider a network of agents, each with its own private cost consisting of the sum of two possibly nonsmooth convex functions, one of which is composed with a linear operator. At every iteration each agent performs local calculations and can only communicate with its neighbors. The goal is to minimize the aggregate of the private cost functions and reach a consensus over a graph. We propose a primal-dual algorithm based on \emph{Asymmetric Forward-Backward-Adjoint} (AFBA), a new operator splitting technique introduced recently by two of the authors. Our algorithm includes the method of Chambolle and Pock as a special case and has linear convergence rate when the cost functions are piecewise linear-quadratic. We show that our distributed algorithm 
is easy to implement without the need to perform matrix inversions or inner loops.  We demonstrate through computational experiments how selecting the parameter of our algorithm can lead to larger step sizes and yield better performance. 
\end{abstract}

\section{Introduction}

In this paper we deal with the distributed solution of the following
optimization problem:
\begin{equation}\label{eq:Problem}
\underset{x\in\Re^n}{\minimize} \ \sum_{i=1}^{N} f_i(x) + g_i(C_i x) 
\end{equation}
where for $i=1,\ldots,N$, $C_i$ is a linear operator, $f_i$ and $g_i$ are proper closed convex and possibly nonsmooth functions. We further assume that the \emph{proximal mappings} associated with $f_i$ and $g_i$
are efficiently computable \cite{combettes2011proximal}. In a more general case we can include another continuously differentiable term with Lipschitz-continuous gradient in 
\eqref{eq:Problem} and use \cite[Algorithm 3]{AFBA2016} that includes the algorithm of V\~u and Condat \cite{vu2013splitting,condat2013primal} as special case. We do not pursue this here for clarity of exposition. 

Problems of this form appear in several application fields. In a distributed model predictive control setting, $f_i$ can represent individual finite-horizon costs for each agent, $C_i$ model the linear dynamics of each agent and possibly coupling constraints that are split through the introduction of extra variables, and $g_i$ model state and input constraints. 

In machine learning
and statistics the $C_i$ are feature matrices and functions $g_i$ measure the \emph{fitting}
of a predicted model with the observed data, while the $f_i$ is \emph{regularization} terms that enforces
some prior knowledge in the solution (such as sparsity, or belonging to a certain constraint set).
For example if $g_i$ is the so-called hinge loss and $f_i = \tfrac{\lambda}{2}\|\cdot\|_2^2$, for some $\lambda > 0$,
then one recovers the standard SVM model. If instead $f_i = \lambda\|\cdot\|_1$ then
one recovers the $\ell_1$-norm SVM problem \cite{zhu20041}. 

Clearly problem \eqref{eq:Problem} can be solved in a centralized fashion, when all the
data of the problem (functions $f_i$, $g_i$ and matrices $C_i$, for all $i\in\{1,\ldots,N\}$) are available at one
computing node. When this is the case one might formulate and solve the aggregated problem
\begin{equation*}
\underset{x\in\Re^n}{\minimize} \ f(x) + g(C x),
\end{equation*}
for which algorithms are available \cite{chambolle2011first,AFBA2016,briceno2011monotone+}. 
However, such a centralized approach is not realistic in many scenarios. For example, suppose
that $g_i(C_i x)$  models least-squares terms and $C_1,\ldots,C_N$ are very large features matrices.
Then collecting $C_1,\ldots,C_N$ into a single computer may be infeasible due to communication costs,
or even worse they may not fit into the computer's memory. Furthermore, the exchange of such information
may not be possible at all due to privacy issues.

Our goal is therefore to solve problem \eqref{eq:Problem} in a distributed
fashion. Specifically, we consider a connected network of $N$ computing agents,
where the $i$-th agent is able to compute proximal mappings of $f_i$, $g_i$, and matrix vector products with $C_i$ (and its adjoint operator). We want all the agents
to iteratively converge to a \emph{consensus} solution to \eqref{eq:Problem}, and to do so
by only exchanging variables among neighbouring nodes, i.e, no centralized computations (\emph{i.e.}, existence of a fusion center) are needed
during the iterations. 

To do so, we will propose a solution based on the recently introduced \emph{Asymmetric Forward-Backward-Adjoint} (AFBA) splitting method \cite{AFBA2016}. This new splitting technique solves monotone inclusion problems involving three operators, however, in this work we will focus on a special case that involves two terms. Specifically, we develop a distributed algorithm which is based on a special case of AFBA applied to the monotone inclusion corresponding to the primal-dual optimality conditions of a suitable \emph{graph splitting} of~\eqref{eq:Problem}. Our algorithm involves a nonnegative parameter $\theta$ which serves as a tuning knob that allows to recover different algorithms. In particular, the algorithm of \cite{chambolle2011first} is recovered in the special case when $\theta=2$. We demonstrate how tuning this parameter affects the stepsizes and ultimately the convergence rate of the algorithm.  

Other algorithms have been proposed for solving problems similar to \eqref{eq:Problem} in a distributed way.
As a reference framework, all algorithms aim at solving in a distributed way the problem
\begin{equation*}
\underset{x\in\Re^n}{\minimize} \ \sum_{i=1}^N F_i(x).
\end{equation*}
In \cite{nedic2009distributed} a distributed subgradient method is proposed, and in \cite{duchi2012dual} this
idea is extended to the projected subgradient method.
More recently, several works focused on the use of
ADMM for distributed optimization.
In \cite{Boyd2010a} the generic ADMM for consensus-type problems is illustrated.
A drawback of this approach is that at every iteration the agents must solve a complicated subproblem that might
require an inner iterative procedure. In \cite{Parikh2014} another formulation is given for the
case where $F_i = f_i+g_i$, and only proximal mappings with respect to $f_i$ and $g_i$ are separately computed in each node.
Still, when either $f_i$ or $g_i$ is not separable (such as when they are composed with linear operators) these are
not trivial to compute and may require inner iterative procedures, or factorization of the data matrices involved.
Moreover, in both \cite{Boyd2010a, Parikh2014} a central node is required for accumulating each agents variables at every iteration,
therefore these formulations lead to \emph{parallel} algorithms rather than distributed. In \cite{teixeira2013optimal} the optimal
parameter selection for ADMM is discussed in the case of distributed quadratic programming problems.
In \cite{wei2012distributed, wei20131, makhdoumi2014broadcast}, fully distributed algorithms based on ADMM proposed,
assuming that the proximal mapping of $F_i$ is computable, which is impractical in many cases. In \cite{bianchi2014primal}  the authors propose a variation of the
V{\~u}-Condat algorithm \cite{condat2013primal,vu2013splitting}, having
ADMM as a special case, and show its application to distributed optimization where $F_i = f_i+g_i$, but no composition with a linear operator is involved.
Only proximal operations with respect to $f_i$ and $g_i$ and local exchange of variables (\emph{i.e.}, among neighboring nodes)
is required, and the method is analyzed in an asynchronous setting.

In this paper we deal with the more general scenario of problem \eqref{eq:Problem}.
The main features of our approach, that distinguish it from the related works mentioned above, are:
\begin{enumerate}[(i)]
\item We deal with $F_i$ that is the sum of two possibly nonsmooth functions one of which is composed with a linear operator.
\item Our algorithm only require local exchange of information, i.e., only neighboring
nodes need to exchange local variables for the algorithms to proceed.
\item The iterations involve \emph{direct} operations on the objective terms. Only evaluations of
$\prox_{f_i}$, $\prox_{g^\star_i}$ and matrix-vector products with $C_i$ and $C_i^T$ are involved.
In particular, no inner subproblem needs to be solved iteratively by the computing agents, and no matrix inversions are required.
\end{enumerate}

The paper is organized as follows.
 \Cref{sec:Formulation} is devoted to  a formulation of problem
\eqref{eq:Problem} which is amenable to be solved in a distributed fashion by the proposed methods. In \Cref{sec:Algorithms} we detail how the primal-dual algorithm in \cite[Algorithm 6]{AFBA2016} together with an intelligent change of variables gives rise to distributed iterations. We then  discuss implementation considerations and convergence properties. In \Cref{sec:Simulations} we illustrate some numerical results for several values of the constant $\theta$, highlighting the improved performance for $\theta=1.5$.


\section{Problem Formulation} \label{sec:Formulation}
Consider problem~\eqref{eq:Problem}
under the following assumptions:
\begin{ass}
For $i=1,\ldots,N$:
\begin{enumerate}[(i)]
	\item $C_i:\Re^n\to\Re^{r_i}$ are linear operators.
\item $f_i:\Re^n\to\Rinf$, $g_i:\Re^{r_{i}}\to\Rinf$ are proper closed convex functions, where $\Rinf=\Re\cup\{\infty\}$.  
\item The set of minimizers of~\eqref{eq:Problem}, denoted by $S^\star$, is nonempty. 
\end{enumerate}
\end{ass}
We are interested in solving problem~\eqref{eq:Problem} in a distributed fashion. Specifically, let $G=(V,E)$ be an undirected graph over the vertex set $V=\{1,\ldots,N\}$ with edge set $E\subset V\times V$. 
It is assumed that each node $i\in V$ is associated with a separate agent, and each agent maintains its own cost components $f_i$, $g_i$, $C_i$ which are assumed to be private, and its own opinion of the solution $x_i\in\Re^n$. The graph imposes communication constraints over agents. In particular, agent $i$ can communicate directly only with its neighbors $j\in\mathcal{N}_i=\{j\in V\ |\ (i,j)\in E\}$. We make the following assumption. 
\begin{ass} \label{ass:conctd}
Graph $G$ is connected.
\end{ass}
With this assumption, we reformulate the problem as
\begin{align*}
\underset{\bm{x}\in\Re^{Nn}}{\minimize}&\quad \sum_{i=1}^N f_i(x_i)+g_i(C_ix_i)\\
\stt&\quad x_i=x_j\qquad (i,j)\in E
\end{align*}
where  $\bm{x}=(x_1,\ldots,x_N)$. 
Associate any orientation to the unordered edge set $E$. Let $M=|E|$ and $B\in\Re^{N\times M}$ be the \emph{oriented node-arc incidence matrix}, where each column is associated with an edge $(i,j)\in E$ and has $+1$  and $-1$ in the $i$-th and $j$-th entry, respectively. Notice that the sum of each column of $B$ is equal to $0$. Let $d_i$ denote the degree of a given vertex, that is, the number of
vertices that are adjacent to it. We have $B B^\top=\mathcal{L}\in\Re^{N\times N}$, where $\mathcal{L}$ is the graph \emph{Laplacian} of $G$, i.e.,
$$\mathcal{L}_{ij}=\begin{cases}d_i&\textrm{ if } i=j,\\
-1&\textrm{ if } i\neq j \textrm{ and node $i$ is adjacent to node $j$},\\
0&\textrm{ otherwise}.
\end{cases}$$

Constraints $x_i=x_j$, $(i,j)\in E$ can be written in compact form as $A\bm{x}=0$, where $A=B^\top\otimes I_n\in\Re^{Mn\times Nn}$. Therefore, the problem is expressed as
\begin{equation} \label{eq:primal}
\underset{\bm{x}\in\Re^{Nn}}{\minimize}\quad \sum_{i=1}^N f_i(x_i)+g_i(C_ix_i)+\delta_{\{0\}}(A\bm{x}),
\end{equation}
 where $\delta_{X}$ denotes the indicator function of a closed nonempty convex set, $X$. The dual problem is:
 \begin{equation} \label{eq:dual}
 \underset{\begin{subarray}{c}
   y_i\in\Re^{r_i} \\
   w\in\Re^{Mn}
   \end{subarray}}{\minimize}\quad \sum_{i=1}^N f_i^*(-A^\top_i w-C_i^\top y_i)+g_i^*(y_i),
 \end{equation}
 where $q^*$ denotes the Fenchel conjugate of a function $q$ and $A_i\in\Re^{Mn\times n}$  are the block columns of $A$. 
Let $\partial q$ denote the subdifferential of a convex function $q$. The primal-dual optimality conditions are 
\begin{equation} \label{eq:-1}
\begin{cases}
0\in\partial f_{i}(x_{i})+C_{i}^{\top}y_{i}+A_{i}^{\top}w, & i=1,\ldots,N\\
C_{i}x_{i}\in\partial g_{i}^{*}(y_{i}) & i=1,\ldots,N\\
\sum_{i=1}^{N}A_{i}x_{i}=0,
\end{cases}
\end{equation}
where $w\in\Re^{Mn}$, $y_i\in\Re^{r_i}$, for $i=1,\ldots,N$. 
The following condition will be assumed throughout the rest of the paper.
\begin{ass} \label{ass:exst-prmdul}
There exist $x_i\in \ri \dom f_i$ such that $C_i x_i\in \ri \dom g_i$, $i=1,\ldots,N$ and $\sum_{i=1}^{N}A_i x_i=0$\footnote{$\dom f$ denotes the domain of function $f$ and $\ri C$ is the relative interior of the set $C$.}. 
\end{ass}
This assumption guarantees that the set of solutions to~\eqref{eq:-1} is nonempty (see  \cite[Proposition 4.3(iii)]{combettes2012primal}). If $(\boldsymbol{x}^\star,\boldsymbol{y}^\star,w^\star)$ is a solution to~\eqref{eq:-1}, then  $\boldsymbol{x}^\star$ is a solution to the primal problem~\eqref{eq:primal} and $(\boldsymbol{y}^\star,w^\star)$ to its dual~\eqref{eq:dual}. 
\section{Distributed Primal-Dual Algorithms} \label{sec:Algorithms}

In this section we provide the main distributed algorithm that is based on \emph{Asymmetric Forward-Backward-Adjoint} (AFBA), a new operator splitting technique introduced recently \cite{AFBA2016}. This special case belongs to the class of primal-dual algorithms. The convergence results include both the primal and dual variables and are based on \cite[Propositions 5.4]{AFBA2016}). However, the convergence analysis here focuses on the primal variables for clarity of exposition, with the understanding that a similar error measure holds for the dual variables. 

Our distributed algorithm 
consists of two phases, a local phase and the phase in which each agent  interacts with its neighbors according to the constraints imposed by the communication graph. Each iteration has the advantage of only requiring local matrix-vector products and proximal updates. Specifically, each agent performs $2$  matrix-vector products per iteration and transmits a vector of dimension $n$ to its neighbors. 

Before continuing we recall the definition of Moreau's proximal mapping. Let $U$ be a symmetric positive-definite matrix. 
 The \emph{proximal mapping} of a proper closed convex function $f:\Re^n\to\Rinf$ relative to $\|\cdot\|_U$ is defined by
\begin{equation*}
	\prox^U_f(x) = \argmin_{z\in \Re^n} f(z)+\frac{1}{2} \|x-z\|_U^2,
\end{equation*}
and when the superscript $U$ is omitted the same definition applies with respect to the canonical norm.  

 Let $\boldsymbol{u}=(\boldsymbol{x},\boldsymbol{v})$  where 
 $\boldsymbol{v}=(\boldsymbol{y},w)$ and 
 $\boldsymbol{y}=(y_{1},\ldots,y_{N})$. 
 The optimality conditions in~\eqref{eq:-1}, can be written in the form of the following monotone inclusion: 
 \begin{equation}\label{eq:mon-inclusion}
 0\in D\boldsymbol{u} + M\boldsymbol{u} 
 \end{equation}
 with 
\begin{align}
{D}(\boldsymbol{x},\boldsymbol{y},w)&=(\partial\boldsymbol{f}(\boldsymbol{x}), \partial\boldsymbol{g}^{*}(\boldsymbol{y}),0),\label{eq:Aopt}
\end{align}
and
\begin{equation*}
{M}=\left[\begin{array}{ccc}
  0 & \boldsymbol{C}^{\top} & {A}^{\top}\\
  -\boldsymbol{C} & 0 & 0\\
  -{A} & 0 & 0
  \end{array}\right],
\end{equation*}
where $\boldsymbol{f}(\boldsymbol{x})=\sum_{i=1}^{N}f_{i}(x_{i})$
 , 
  $\boldsymbol{g}^{*}(\boldsymbol{y})=\sum_{i=1}^{N}g_{i}^*(y_{i})$, $\boldsymbol{C}=\blkdiag(C_1,\ldots,C_N)$.  Notice that ${A}x=\sum_{i=1}^{N}A_{i}x_{i}$, ${A}^\top w=(A_1^\top w,\ldots,A_N^\top w)$. 
The operator $D+M$ is maximally monotone \cite[Proposition 20.23, Corollary 24.4(i)]{bauschke2011convex}.

Monotone inclusion~\eqref{eq:mon-inclusion}, \emph{i.e.}, the primal-dual optimality conditions~\eqref{eq:-1}, is solved by applying \cite[Algorithm 6]{AFBA2016}.  This results in the following iteration: 
\begin{subequations} \label{alg:5}
	\begin{align}
		{\boldsymbol{x}}^{k+1} & =\prox^{\Sigma^{-1}}_{ f}(\boldsymbol{x}^{k}-\Sigma C^{\top}\boldsymbol{y}^k-\Sigma A^{\top}w^{k}) \label{eq:a}\\
		\bar{\boldsymbol{y}}^{k} & =\prox^{\Gamma^{-1}}_{ g^{*}}(\boldsymbol{y}^{k}+\Gamma C(\theta{\boldsymbol{x}}^{k+1}+(1-\theta){\boldsymbol{x}}^{k}))\\
		\bar{w}^{k} & =w^{k}+\Pi A(\theta{\boldsymbol{x}}^{k+1}+(1-\theta){\boldsymbol{x}}^{k}) \label{eq:c}\\
		\boldsymbol{y}^{k+1} & =\bar{\boldsymbol{y}}^{k}+(2-\theta)\Gamma C({\boldsymbol{x}}^{k+1}-{\boldsymbol{x}}^{k}) \label{eq:d}\\
		w^{k+1} &= \bar{w}^k + (2-\theta)\Pi A({\boldsymbol{x}}^{k+1}-{\boldsymbol{x}}^{k}) \label{eq:e}
	\end{align}
\end{subequations}
where matrices $\Sigma,\Gamma,\Pi$ play the rule of stepsizes and are assumed to be positive definite. The iteration~\eqref{alg:5} can not be implemented in a distributed fashion because the dual vector $w$ consists of $M$ blocks corresponding to the edges. The key idea that allows distributed computations is to introduce the sequence
 \begin{equation}\label{eq:rho}
 (\rho^k_{i})_{k\in\Nn}=(A_{i}^{\top}w^{k})_{k\in\Nn}, \quad \textrm{for} \quad i=1,\ldots,N.
 \end{equation}
 This transformation replaces the stacked edge vector $w^k$ with corresponding node vectors $\rho_i$. More compactly, letting $\boldsymbol{\rho}^k=(\rho^k_{1},\ldots,\rho^k_{N})$, it follows from~\eqref{eq:c} and \eqref{eq:e} that 
\begin{equation}
\boldsymbol{\rho}^{k+1}=\boldsymbol{\rho}^{k}+ A^\top \Pi A(2{\boldsymbol{x}}^{k+1}-{\boldsymbol{x}}^{k}),
\end{equation}
 where $A^\top \Pi A$ is the \emph{weighted graph Laplacian}. Since $w^k$ in~\eqref{eq:a} appear as $A^\top w^k$ we can rewrite the iteration: 
 	\begin{align*}
 		{\boldsymbol{x}}^{k+1} & =\prox^{\Sigma^{-1}}_{ f}(\boldsymbol{x}^{k}-\Sigma C^{\top}\boldsymbol{y}^k-\Sigma \boldsymbol{\rho}^{k})\\
 		\bar{\boldsymbol{y}}^{k} & =\prox^{\Gamma^{-1}}_{ g^{*}}(\boldsymbol{y}^{k}+\Gamma C(\theta{\boldsymbol{x}}^{k+1}+(1-\theta){\boldsymbol{x}}^{k})) \\
 		\boldsymbol{y}^{k+1} & =\bar{\boldsymbol{y}}^{k}+(2-\theta)\Gamma C({\boldsymbol{x}}^{k+1}-{\boldsymbol{x}}^{k}) \\
 		\boldsymbol{\rho}^{k+1}&=\boldsymbol{\rho}^{k}+ A^\top \Pi A(2{\boldsymbol{x}}^{k+1}-{\boldsymbol{x}}^{k})
 	\end{align*}
Set 
\begin{align*}
	\Sigma&=\blkdiag\left(\sigma_1 I_n,\ldots,\sigma_N I_n\right),\\
	\Gamma&=\blkdiag\left(\tau_1 I_{r_1},\ldots,\tau_N I_{r_N}\right),\\
	\Pi&=\blkdiag\left(\pi_1 I_n,\ldots,\pi_M I_n\right),
\end{align*}
where $\sigma_i>0,\tau_i>0$ for $i=1,\ldots,N$ and $\pi_l>0$ for $l=1,\ldots,M$. Consider a bijective mapping between $l=1,\ldots,M$ and unordered pairs $(i,j)\in E$ such that $\kappa_{i,j}=\kappa_{j,i}=\pi_l$.
  Notice that $\pi_l$ for $l=1,\ldots,M$ are step sizes to be selected by the algorithm and can be viewed as weights for the edges. Thus, iteration~\eqref{alg:5} gives rise to our distributed algorithm:  
  \begin{algorithm}[H]
  	\caption{ 
  	}
  	\label{Algorithm-5}
  	\begin{algorithmic}
  		\item\textbf{Inputs:} $\sigma_i>0$, $\tau_i>0$, $\kappa_{i,j}>0$ for $j\in\mathcal{N}_{i}$, $i=1,\ldots,N$, $\theta\in[0,\infty[$, \textrm{initial values $x_i^0\in\Re^n$, $y_i^0\in\Re^{r_i}$, $\rho_i^0\in\Re^n$}.
  		\For{$k=1,\ldots$}
  		\For {each agent $i=1,\ldots,N$}
  		\item{\hspace{0.9cm} Local steps:}
  		\State ${x}_{i}^{k+1}=\prox_{\sigma_{i}f_{i}}(x_{i}^{k}-\sigma_{i}\rho_{i}^{k}-\sigma_{i}C_i^\top y_{i}^{k})$
  		\State $\bar{y}_{i}^{k}=\prox_{\tau_{i}g_{i}^{*}}(y_{i}^{k}+\tau_{i}C_{i}(\theta x_i^{k+1}+(1-\theta)x_i^{k}))$
  		\State $y_i^{k+1} = \bar{y}_{i}^{k} + \tau_i(2-\theta)C_i (x_i^{k+1}-x_i^k)$
  		\State $u_i^k=2x_i^{k+1}-x_i^k$
  		\item {\hspace{0.9cm} Exchange of information with neighbors:}
  		\State $\rho_{i}^{k+1}=\rho_{i}^{k}+\sum_{j\in\mathcal{N}_{i}}\kappa_{i,j}(u_{i}^{k}-u_{j}^{k})$ 
  		\EndFor
  		\EndFor
  	\end{algorithmic}
  \end{algorithm} 
  Notice that each agent $i$ only requires $u_j^k\in\Re^n$ for $j\in\mathcal{N}_{i}$ during the communication  phase. 
  Before proceeding with convergence results, we define the following for simplicity of notation:
\begin{equation*}
\begin{split}
 \bar{\sigma}&=\max\{\sigma_1,\ldots,\sigma_N\},\\
 \bar{\tau}&=\max\{\tau_1,\ldots,\tau_N,\pi_1,\ldots,\pi_M\},\\
L&= \mathcal{L}\otimes I_n+\boldsymbol{C}^\top \boldsymbol{C},\ \textrm{where }\mathcal{L}\ \textrm{is the graph Laplacian.}
 \end{split}
\end{equation*}
It must be noted that the results in this section only provide choices of parameters that are sufficient for convergence. They can be selected much less conservatively by formulating and solving sufficient conditions that they must satisfy  as \emph{linear matrix inequalities} (LMIs). Due to lack of space we do not pursue this direction further, instead we plan to consider it in an extended version of this work.

\begin{thm} \label{thm:Algorithm-5}
	Let \Cref{ass:conctd,ass:exst-prmdul} hold true. Consider the sequence $(\boldsymbol{x}^{k})_{k\in\Nn}=(x_1^k,\ldots,x_N^k)_{k\in\Nn}$ generated by \Cref{Algorithm-5}. Assume the maximum stepsizes, \emph{i.e.}, $\bar{\sigma}$ and $\bar{\tau}$ defined above, are positive and satisfy     
	\begin{equation} \label{eq:alg5}
	\bar{\sigma}^{-1}-\bar{\tau}(\theta^2-3\theta+3)\|L\|>0,
	\end{equation}
	for a fixed value of $\theta\in[0,\infty[$. 
	Then 
	 the sequence $(\boldsymbol{x}^{k})_{k\in\Nn}$ converges to $(x^{\star},\ldots,x^{\star})$ for some $x^\star\in S^\star$. Furthermore, if $\theta=2$ the strict inequality~\eqref{eq:alg5} is replaced with  $\bar{\sigma}^{-1}-\bar{\tau}\|L\|\geq0$.
	
\end{thm}

\begin{proof}
	\Cref{Algorithm-5} is an implementation of \cite[Algorithm 6]{AFBA2016}. 
	Thus convergence of $(\bld{x}^k)_{k\in\Nn}$ to a solution of~\eqref{eq:primal} is implied by \cite[Proposition 5.4]{AFBA2016}. Combining this with \Cref{ass:conctd} yields the result. 
	Notice that in that work the step sizes are assumed to be scalars for simplicity. It is straightforward to adapt the result to the case of diagonal matrices. 
\end{proof}
In \Cref{Algorithm-5} when $\theta=2$, we recover the algorithm of Chambolle and Pock \cite{chambolle2011first}. 
One important observation is that the term $\theta^2-3\theta+3$ in~\eqref{eq:alg5} is always positive and achieves its minimum at $\theta=1.5$. This is a choice of interest for us since it results in larger stepsizes, ${\sigma}_i,\tau_i, \kappa_{i,j}$, and consequently better performance as we observe in numerical simulations. 

Next, we provide easily verifiable conditions for $f_i$ and $g_i$, under which  linear convergence of the iterates can be established. We remark that these are just sufficient and certainly less conservative conditions can be provided but are omitted for clarity of exposition. Let us first recall the following definitions from  \cite{rockafellar2009variational,dontchev2009implicit}: 
\begin{deff}[Piecewise Linear-Quadratic]
	A function $f:\Re^n\to\Rinf$ is called piecewise linear-quadratic (PLQ) if it's domain can be represented as union of finitely many polyhedral sets, relative to each of which $f(x)$ is given by an expression of the form $\frac{1}{2}x^\top Qx+d^\top x+c$, 
	for some $c\in\Re$, $d\in\Re^n$, and $D\in\Re^{n\times n}$. 
	\end{deff}
	The class of  piecewise linear-quadratic functions has been much studied and has many desirable properties (see \cite[Chapter 10 and 11]{rockafellar2009variational}). 
Many practical applications involve PLQ functions such as quadratic function, $\|\cdot\|_1$, indicator of polyhedral sets, hinge loss, \emph{etc}. Thus, the $R$-linear convergence rate that we establish in \Cref{thm:conv-rate} holds for a wide range of problems encountered in control, machine learning and signal processing. 
	\begin{deff}[Metric subregularity] \label{metricsub}
		A set-valued mapping $F:\Re^n\rightrightarrows\Re^n$ is metrically subregular at $z$ for $z^\prime$ if $(z,z^\prime)\in\gra F$ and there exists $\eta\in[0,\infty[$, 
		a neighborhood $\mathcal{U}$ of $z$ and $\mathcal{V}$ of $z^\prime$ such that 
		\begin{equation}\label{eqnn:metricsubregularity}
		d(x,F^{-1}z^\prime)\leq \eta d(z^\prime,Fx\cap \mathcal{V})  \;\; \textrm{for all} \; x\in \mathcal{U},
		\end{equation}
		where $\gra F=\{(x,u)|u\in Fx\}$ and $d(\cdot,X)$ denotes the distance from set $X$.  
	\end{deff}
\begin{thm}\label{thm:conv-rate}
	Consider \Cref{Algorithm-5} under the assumptions of \Cref{thm:Algorithm-5}. Assume $f_i$ and $g_i$ for $i=1,\ldots,N$, are piecewise linear-quadratic functions. Then the set valued mapping $T=D+M$  is \emph{metrically subregular} at any $z$ for any $z^\prime$ provided that $(z,z^\prime)\in\gra T$.
	 Furthermore, the sequence $(\boldsymbol{x}^k)_{k\in\Nn}$ converges R-linearly\footnote{The sequence $(x_{n})_{n\in\Nn}$ converges to $x^{\star}$ $R$-linearly if there is a sequence of nonnegative scalars $({v_n})_{n\in\Nn}$ such that  $\|x_{n}-x^{\star}\|\leq v_n$ and $(v_n)_{n\in\Nn}$ converges $Q$-linearly\footnotemark{} to zero.} \footnotetext{The sequence $(x_{n})_{n\in\Nn}$ converges to $x^{\star}$ $Q$-linearly with $Q$-factor given by $\sigma\in]0,1[$, if for $n$ sufficiently large  $\|x_{n+1}-x^{\star}\|\leq\sigma\|x_{n}-x^{\star}\|$ holds.}to $(x^{\star},\ldots,x^{\star})$ for some $x^\star\in S^\star$.  
\end{thm}
\begin{proof}
	Function $\boldsymbol{f}(\boldsymbol{x})=\sum_{i=1}^{N}f_{i}(x_{i})$ is piecewise linear-quadratic since $f_i$ for $i=1,\ldots,N$, are assumed to be PLQ. Similarly, it follows from \cite[Theorem 11.14 (b)]{rockafellar2009variational} that $\boldsymbol{g}^*$ is  piecewise linear-quadratic. The subgradient mapping of a proper closed convex PLQ function is piecewise polyhedral, \emph{i.e.} its graph is the union of finitely many polyhedral sets  \cite[Proposition 12.30 (b)]{rockafellar2009variational}. This shows that $D$ defined in~\eqref{eq:Aopt} is piecewise polyhedral. Since the image of a polyhedral under affine transformation remains piecewise polyhedral, and $M$ is a linear operator, graph of $T=D+M$ is piecewise polyhedral. Consequently, its inverse $T^{-1}$ is piecewise polyhedral. Thus by \cite[Proposition 3H.1]{dontchev2009implicit} the mapping $T^{-1}$ is calm at any $z^\prime$ for any $z$ satisfying $(z^\prime,z)\in \gra T^{-1}$. This is equivalent to the metric subregularity characterization of the operator $T$ \cite[Theorem 3H.3]{dontchev2009implicit}. The second part of the proof follows directly by noting that \cite[Algorithm 6]{AFBA2016} used to derive \Cref{Algorithm-5}  is a special case of \cite[Algorithm 1]{AFBA2016}. Therefore, linear convergence follows from first part of the proof together with \cite[Theorem 3.3]{AFBA2016}. The aforementioned theorem guarantees linear convergence for the stacked vector $\boldsymbol{u}$ in \eqref{eq:mon-inclusion}, however, here we consider the primal variables only.
\end{proof}

\subsection{Special Case} Consider the following problem
\begin{equation}\label{eq:specialprob2}
\underset{x\in\Re^n}{\minimize} \ \sum_{i=1}^{N} f_i(x),
\end{equation}
where $f_i:\Re^n\to\Rinf$ for $i=1,\ldots,N$ are proper closed convex functions. This is a special case of~\eqref{eq:Problem} when  $g_i\circ C_i\equiv0$. Since functions $g_i$ are absent, the dual variables $y_i$ in \Cref{Algorithm-5} vanish and for any choice of $\theta$ the algorithm reduces to: 
\begin{align*}
	x_i^{k+1}&=\prox_{\sigma_{i}f_{i}}(x_{i}^{k}-\sigma_{i}\rho_{i}^{k})\\
	u_i^k&= 2x_i^{k+1}-x_i^k\\
	\rho_{i}^{k+1}&=\rho_{i}^{k}+\sum_{j\in\mathcal{N}_{i}}\kappa_{i,j}({u}_{i}^{k}-{u}_{j}^{k}).
\end{align*}
Thus setting $\theta=1.5$ in \eqref{eq:alg5} to maximize the stepsizes yields $\bar{\sigma}^{-1}-\frac{3\bar{\tau}}{4}\|\mathcal{L}\|>0$, where $\mathcal{L}$ is the graph Laplacian.

\section{Numerical Simulations}\label{sec:Simulations}

We now illustrate experimental results obtained by applying the proposed algorithm to the following problem:  
\begin{equation}\label{eq:Example}\begin{split}
\minimize \ & \lambda\|x\|_1 + \sum_{i=1}^N \tfrac{1}{2}\|D_ix - d_i\|_2^2 
\end{split}\end{equation}
for a positive parameter $\lambda$. This is the ${\ell}_1$ regularized least-squares problem. Problem \eqref{eq:Example} is of the form \eqref{eq:Problem} if we set for $i=1,\ldots,N$
\begin{equation}\begin{split}
f_i(x) &= \tfrac{\lambda}{N}\|x\|_1,\\
g_i(z) &= \tfrac{1}{2}\|z - d_i\|_2^2,\\
C_i &= D_i
\end{split}\end{equation}
where $D_i\in\Re^{m_i\times n}$,  $d_i\in\Re^{m_i}$. 
For the experiments we used graphs of $N=50$ computing agents, generated randomly according to the Erd\H{o}s-Renyi model, with parameter $p=0.05$.
In the experiments we used $n=500$ and generated
$D_i$ randomly with normally distributed entries,
with $m_i = 50$ for all $i=1,\ldots,N$. Then we generated vector $d_i$
starting from a known solution for the problem and ensuring {
$\lambda<0.1 \|\sum_i^N D_i^\top d_i\|_{\infty}$}.

For the stepsize parameters we set
$\sigma_i = \bar{\sigma}$, $\tau_i = \bar{\tau}$, for all $i=1,\dots,N$, and $\kappa_{i,j} = \kappa_{j,i} = \bar{\tau}$ for all edges $(i,j)\in E$, such that \eqref{eq:alg5} is satisfied. In order to have a fair comparison we selected $\bar{\sigma}=\alpha/\|L\|$ and $\bar{\tau}=0.99/(\alpha(\theta^2-3\theta+3))$ with $\alpha=20$  which  was set empirically based on better performance of all the algorithms. 

The results are illustrated in \Cref{fig1}, for several values of $\theta$, where the distribution of the number of communication rounds required by the algorithms to reach a relative error of $10^{-6}$ is reported. In \Cref{fig2} the convergence of algorithms is illustrated in one of the instances.  It
should be noted that the algorithm of Chambolle and Pock, that corresponds to $\theta=2$, is generally slower than the case $\theta=1.5$. This is mainly due to the larger stepsize parameters guaranteed by \Cref{thm:Algorithm-5}.

\begin{figure}
	\begin{minipage}{\columnwidth}
		\hspace*{0.2cm}{\centering\includegraphics{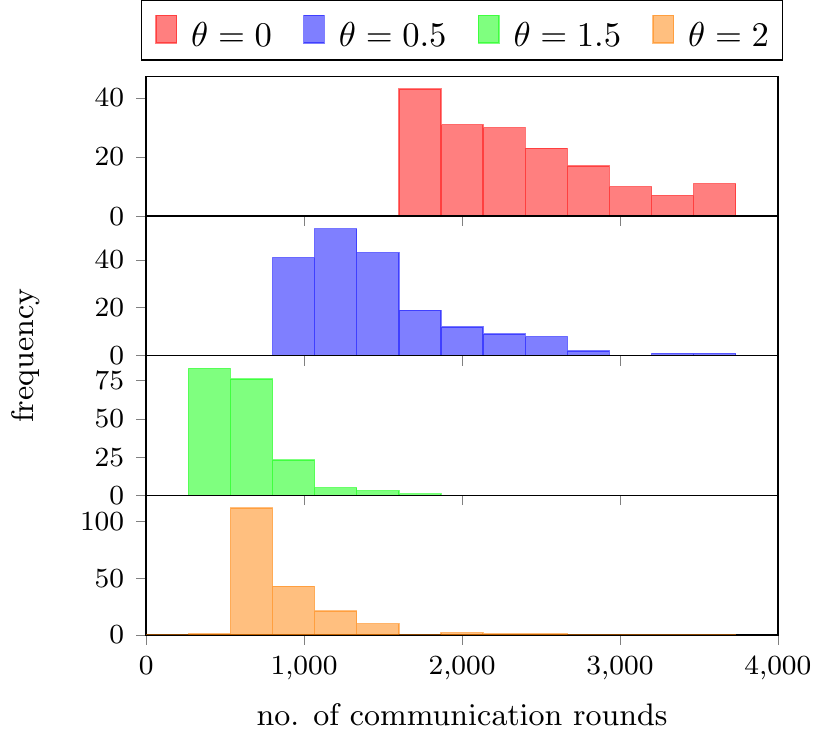}}
	\end{minipage}
		\caption{Distribution of the number of communication rounds required by the algorithms to achieve a relative error of $10^{-6}$, for fixed data and $200$ randomly generated Erd\H{o}s-Renyi graphs, with parameter $p=0.05$.}
		\label{fig1}
\end{figure}

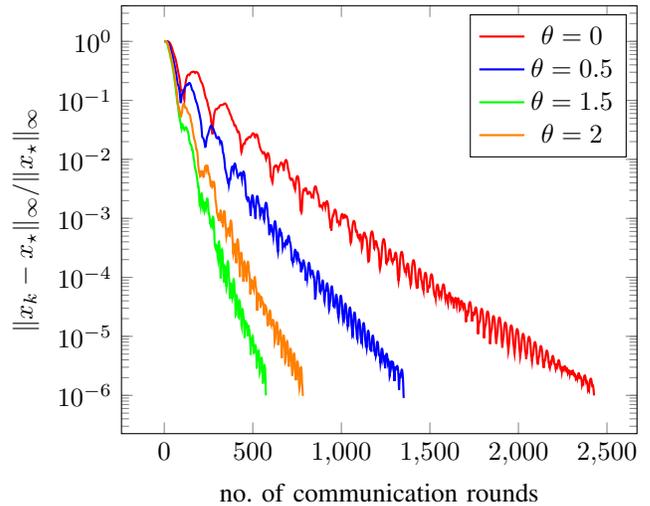
\begin{figure}
\center
\begin{tikzpicture}
\begin{axis}[ymode=log,
xlabel={no. of communication rounds},
ylabel={$\|x_k-x_\star\|_\infty/\|x_\star\|_\infty$}
]
\addplot[color=red,thick]
	file {plotdata/alg1_convergence.dat};
	\addplot[color=blue,thick]
	file {plotdata/alg2_convergence.dat};
	\addplot[color=green,thick]
	file {plotdata/alg3_convergence.dat};
	\addplot[color=orange,thick]
	file {plotdata/alg4_convergence.dat};
	\legend{$\theta=0$, $\theta=0.5$, $\theta=1.5$, $\theta=2$}
\end{axis}
\end{tikzpicture}
\caption{Convergence of the relative error for the algorithms, in one of the considered instances.}
\label{fig2}
\end{figure}

\section{Conclusions}
In this paper we illustrated how the recently proposed Asymmetric Forward-Backward-Adjoint
splitting method (AFBA)  can be used for
solving distributed optimization problems where a set of $N$ computing agents, connected
in a graph, need to minimize the sum of $N$ functions. The resulting
\Cref{Algorithm-5} only involves
local exchange of variables (i.e., among neighboring nodes) and therefore no central
authority is required to coordinate them. Moreover, the single nodes only require
direct computations of the objective terms, and do not need to perform inner iterations
and matrix inversions. Numerical experiments highlight that \Cref{Algorithm-5} performs generally better
 when the parameter $\theta$ is set equal to $1.5$ in order to achieve the largest stepsizes.
Future investigations on this topic include the study of how the topological structure
of the graph underlying the problem affects the convergence rate of the proposed methods, as well as problem preconditioning in a distributed fashion. Developing  asynchronous versions of the algorithm 
is another important future research direction.  

\bibliographystyle{IEEEtran}
\bibliography{IEEEabrv,refrences}

\begin{thebibliography}{10}
\providecommand{\url}[1]{#1}
\csname url@samestyle\endcsname
\providecommand{\newblock}{\relax}
\providecommand{\bibinfo}[2]{#2}
\providecommand{\BIBentrySTDinterwordspacing}{\spaceskip=0pt\relax}
\providecommand{\BIBentryALTinterwordstretchfactor}{4}
\providecommand{\BIBentryALTinterwordspacing}{\spaceskip=\fontdimen2\font plus
\BIBentryALTinterwordstretchfactor\fontdimen3\font minus
  \fontdimen4\font\relax}
\providecommand{\BIBforeignlanguage}[2]{{%
\expandafter\ifx\csname l@#1\endcsname\relax
\typeout{** WARNING: IEEEtran.bst: No hyphenation pattern has been}%
\typeout{** loaded for the language `#1'. Using the pattern for}%
\typeout{** the default language instead.}%
\else
\language=\csname l@#1\endcsname
\fi
#2}}
\providecommand{\BIBdecl}{\relax}
\BIBdecl

\bibitem{combettes2011proximal}
P.~L. Combettes and I.-C. Pesquet, ``Proximal splitting methods in signal
  processing,'' in \emph{Fixed-point algorithms for inverse problems in science
  and engineering}.\hskip 1em plus 0.5em minus 0.4em\relax Springer, 2011, pp.
  185--212.

\bibitem{AFBA2016}
P.~Latafat and P.~Patrinos, ``Asymmetric forward-backward-adjoint splitting for
  solving monotone inclusions involving three operators,'' \emph{arXiv preprint
  arXiv:1602.08729}, 2016.

\bibitem{vu2013splitting}
B.~C. V{\~u}, ``A splitting algorithm for dual monotone inclusions involving
  cocoercive operators,'' \emph{Advances in Computational Mathematics},
  vol.~38, no.~3, pp. 667--681, 2013.

\bibitem{condat2013primal}
L.~Condat, ``A primal-dual splitting method for convex optimization involving
  {L}ipschitzian, proximable and linear composite terms,'' \emph{Journal of
  Optimization Theory and Applications}, vol. 158, no.~2, pp. 460--479, 2013.

\bibitem{zhu20041}
J.~Zhu, S.~Rosset, T.~Hastie, and R.~Tibshirani, ``1-norm support vector
  machines,'' \emph{Advances in neural information processing systems},
  vol.~16, no.~1, pp. 49--56, 2004.

\bibitem{chambolle2011first}
A.~Chambolle and T.~Pock, ``A first-order primal-dual algorithm for convex
  problems with applications to imaging,'' \emph{Journal of Mathematical
  Imaging and Vision}, vol.~40, no.~1, pp. 120--145, 2011.

\bibitem{briceno2011monotone+}
L.~M. Brice{\~n}o-Arias and P.~L. Combettes, ``A monotone {+} skew splitting
  model for composite monotone inclusions in duality,'' \emph{SIAM Journal on
  Optimization}, vol.~21, no.~4, pp. 1230--1250, 2011.

\bibitem{nedic2009distributed}
A.~Nedi{\'c} and A.~Ozdaglar, ``Distributed subgradient methods for multi-agent
  optimization,'' \emph{IEEE Transactions on Automatic Control}, vol.~54,
  no.~1, pp. 48--61, 2009.

\bibitem{duchi2012dual}
J.~C. Duchi, A.~Agarwal, and M.~J. Wainwright, ``Dual averaging for distributed
  optimization: convergence analysis and network scaling,'' \emph{IEEE
  Transactions on Automatic control}, vol.~57, no.~3, pp. 592--606, 2012.

\bibitem{Boyd2010a}
S.~Boyd, N.~Parikh, E.~Chu, B.~Peleato, and J.~Eckstein, ``{Distributed
  Optimization and Statistical Learning via the Alternating Direction Method of
  Multipliers},'' \emph{Foundations and Trends in Machine Learning}, vol.~3,
  no.~1, pp. 1--122, 2010.

\bibitem{Parikh2014}
N.~Parikh and S.~Boyd, ``{Block splitting for distributed optimization},''
  \emph{Mathematical Programming Computation}, vol.~6, no.~1, pp. 77--102,
  2014.

\bibitem{teixeira2013optimal}
A.~Teixeira, E.~Ghadimi, I.~Shames, H.~Sandberg, and M.~Johansson, ``Optimal
  scaling of the {ADMM} algorithm for distributed quadratic programming,'' in
  \emph{IEEE 52nd Annual Conference on Decision and Control (CDC)}, 2013, pp.
  6868--6873.

\bibitem{wei2012distributed}
E.~Wei and A.~Ozdaglar, ``Distributed alternating direction method of
  multipliers,'' in \emph{IEEE 51st Annual Conference on Decision and Control
  (CDC)}, 2012, pp. 5445--5450.

\bibitem{wei20131}
------, ``On the {$O(1/k)$} convergence of asynchronous distributed alternating
  direction method of multipliers,'' in \emph{IEEE Global Conference on Signal
  and Information Processing (GlobalSIP), 2013 IEEE}, 2013, pp. 551--554.

\bibitem{makhdoumi2014broadcast}
A.~Makhdoumi and A.~Ozdaglar, ``Broadcast-based distributed alternating
  direction method of multipliers,'' in \emph{52nd Annual Allerton Conference
  on Communication, Control, and Computing (Allerton)}.\hskip 1em plus 0.5em
  minus 0.4em\relax IEEE, 2014, pp. 270--277.

\bibitem{bianchi2014primal}
P.~Bianchi and W.~Hachem, ``A primal-dual algorithm for distributed
  optimization,'' in \emph{IEEE 53rd Annual Conference on Decision and Control
  (CDC)}, Dec 2014, pp. 4240--4245.

\bibitem{combettes2012primal}
P.~L. Combettes and J.-C. Pesquet, ``Primal-dual splitting algorithm for
  solving inclusions with mixtures of composite, {L}ipschitzian, and
  parallel-sum type monotone operators,'' \emph{Set-Valued and variational
  analysis}, vol.~20, no.~2, pp. 307--330, 2012.

\bibitem{bauschke2011convex}
H.~H. Bauschke and P.~L. Combettes, \emph{Convex analysis and monotone operator
  theory in {H}ilbert spaces}.\hskip 1em plus 0.5em minus 0.4em\relax Springer
  Science \& Business Media, 2011.

\bibitem{rockafellar2009variational}
R.~T. Rockafellar and R.~J.-B. Wets, \emph{Variational analysis}.\hskip 1em
  plus 0.5em minus 0.4em\relax Springer Science \& Business Media, 2009, vol.
  317.

\bibitem{dontchev2009implicit}
A.~L. Dontchev and R.~T. Rockafellar, ``Implicit functions and solution
  mappings,'' \emph{Springer Monographs in Mathematics. Springer}, vol. 208,
  2009.

\end{thebibliography}

\end{document}